\documentclass[12pt,a4paper]{amsart}
\usepackage[T1]{fontenc}
\usepackage[utf8]{inputenc}
\usepackage{amsmath}
\usepackage{amssymb}
\usepackage{amsthm}
\usepackage{amsfonts}
\usepackage{mathrsfs}
\usepackage{color}
\usepackage[colorlinks,linkcolor=blue,anchorcolor=green,citecolor=red]{hyperref}
\usepackage{cleveref}
\usepackage[margin=30mm]{geometry}
\usepackage{esint}

\makeatletter
\renewcommand\normalsize{%
  \@setfontsize\normalsize{13.2pt}{19pt}%
  \abovedisplayskip 10pt plus 2pt minus 4pt%
  \belowdisplayskip 10pt plus 2pt minus 4pt%
  \abovedisplayshortskip 6pt plus 2pt%
  \belowdisplayshortskip 6pt plus 2pt minus 2pt}
\renewcommand\small{\@setfontsize\small{11.4pt}{15pt}}
\renewcommand\footnotesize{\@setfontsize\footnotesize{9.5pt}{12pt}}
\makeatother
\normalsize

\numberwithin{equation}{section}

\newtheorem{theorem}{Theorem}[section]
\newtheorem{lemma}[theorem]{Lemma}
\newtheorem{proposition}[theorem]{Proposition}
\newtheorem{corollary}[theorem]{Corollary}
\theoremstyle{definition}
\newtheorem{definition}{Definition}[section]
\theoremstyle{remark}
\newtheorem{remark}[theorem]{Remark}
\newtheorem*{discussion}{Discussion of the hypotheses}
\newtheorem*{conjecture}{Landsberg--Berwald conjecture}

\title{The Rigidity of the Closed Three-Dimensional Regular Landsberg Metrics}
\author{Jianyu Mao}
\address{School of Mathematical Sciences, East China Normal University, Shanghai, China}
\email{1430658938@qq.com}
\author{Linfeng Zhou}
\address{School of Mathematical Sciences, East China Normal University, Shanghai, China}
\email{lfzhou@math.ecnu.edu.cn}
\thanks{This work was supported by the Fundamental and Interdisciplinary Disciplines Breakthrough Plan of the Ministry of Education of China (JYB2025XDXM112).}
\date{}
\dedicatory{}

\begin{document}
\maketitle

\begin{abstract}
The Landsberg--Berwald conjecture asks whether every regular Landsberg metric is Berwald. We settle this conjecture for closed three-dimensional manifolds: every smooth strongly convex regular Landsberg metric on such a manifold is Berwald, without any reversibility assumption. Equivalently, there are no regular Landsberg ``unicorns'' on closed three-manifolds; this is a reformulation of the same conjecture, rather than a second independent conjecture. The proof combines a vanishing theorem for commuting Codazzi cubic tensors on closed surfaces, a rigidity theorem for three-dimensional Minkowski norms with constant-curvature indicatrices, and a rank-one argument for the nonlinear curvature. The remaining $R$-quadratic case is settled by compactness along the geodesic flow. The fibrewise constant-curvature theorem is an ingredient in this argument and does not assert the full Laugwitz conjecture.
\end{abstract}

\section{Introduction}

Let $M$ be a smooth manifold. A regular Finsler metric is a continuous function $F:TM\to[0,\infty)$ that is smooth on $T^\circ M=TM\setminus\{0\}$, positive away from the zero section, and positively homogeneous of degree one in each tangent space. We put

\[
E=\frac12F^2,\qquad g_{ij}(x,y)=\frac{\partial^2E}{\partial y^i\partial y^j}(x,y). \tag{1}
\]

The metric is strongly convex if $(g_{ij})$ is positive definite on $T^\circ M$. We make no assumption that $F(x,-y)=F(x,y)$.

A Finsler metric is \emph{Berwald} if the coefficients of its canonical connection depend only on the base point, or equivalently if its canonical parallel translations are linear. It is \emph{Landsberg} if parallel translation preserves the fundamental tensor infinitesimally. Every Berwald metric is Landsberg. The converse is the content of the Landsberg--Berwald conjecture:

\begin{conjecture}
Every regular Landsberg metric is Berwald.
\end{conjecture}

A hypothetical regular Landsberg metric that is not Berwald is often called a \emph{unicorn}. Accordingly, the Landsberg--Berwald conjecture and the so-called unicorn conjecture are two names for the same question, not two separate conjectures. The qualifier \emph{regular} is decisive: almost regular or $y$-local Landsberg structures may fail to be smooth or strongly convex in distinguished nonzero directions, and Asanov's constructions belong to this singular category \cite{ref1}. The conjecture concerns metrics that are smooth and strongly convex on the whole slit tangent bundle.

The problem goes back to Landsberg and Berwald \cite{ref3,ref6} and has accumulated many affirmative results under additional assumptions. A Landsberg metric with vanishing Douglas tensor is Berwald \cite{ref19}. Numata and Shibata obtained rigidity under scalar-curvature hypotheses \cite{ref26,ref30}, while Matsumoto and Shibata treated semi-$C$-reducible metrics \cite{ref24}. Ichijyō related Berwaldness to the linearity of fibre isometries \cite{ref22}. More recently, Shen proved the conjecture for regular $(\alpha,\beta)$-metrics \cite{ref10}, and subsequent work treated general $(\alpha,\beta)$-metrics \cite{ref28,ref29}, $(\alpha_1,\alpha_2)$-metrics \cite{ref27}, and spherically symmetric metrics \cite{ref25}.

Curvature, volume, and completeness assumptions provide another group of partial results. Shen introduced the $R$-quadratic condition \cite{ref9}, and Crampin proved that a complete $R$-quadratic Landsberg metric is Berwald under a boundedness hypothesis \cite{ref4}. Weakly Berwald Landsberg metrics and several conditions formulated through Berwald scalar curvature, invariant volumes, and related contractions were treated by Crampin and Li \cite{ref20,ref21,ref23}. Symmetry of the tangent Minkowski norms gives further rigidity results, including the theorem of Xu and Matveev \cite{ref14}. These theorems offer strong evidence for the conjecture, although their hypotheses are different and do not combine into a proof for arbitrary regular Landsberg metrics.

Bao proposed an approximation program aimed at understanding a regular Landsberg metric through nearby Berwald structures \cite{ref2}. A general proof based on averaging the fundamental tensor was announced by Szabó \cite{ref12} and later withdrawn in a correction \cite{ref13}; Matveev explained the underlying difficulty \cite{ref7}. If a parallel translation $\varphi$ is nonlinear, its differential $d_y\varphi$ varies with the integration point $y$, so the Landsberg isometry identity does not justify replacing it by a fixed linear map in a fibrewise average. In the present paper we avoid this obstruction by working intrinsically with the full Killing algebra of each indicatrix. For a detailed account of these developments and further references, see the survey article \emph{The current state of play in the Landsberg--Berwald problem of Finsler geometry} \cite{ref5}.

The following is the main result.

\begin{theorem}
Let $M$ be a closed smooth three-dimensional manifold, and let $F$ be a smooth strongly convex regular Finsler metric on $M$. If $F$ is Landsberg, then $F$ is Berwald.
\end{theorem}

\begin{discussion}
The word \emph{closed} means compact and without boundary. It is used only in the compactness argument of Theorem 6.1: compactness makes the geodesic flow complete on the unit tangent bundle and makes a certain affine function along each geodesic uniformly bounded. The dimension-three assumption enters earlier and in two different ways. First, each indicatrix is a surface, so its curvature is described by a single Gaussian curvature and its Killing algebra has the dichotomy used in Sections 3 and 4. Second, every nonzero two-form on a three-dimensional vector space is decomposable, which is the key algebraic input in Lemma 5.1. Strong convexity makes the fundamental tensor and every indicatrix metric Riemannian. Smooth regularity on all of $T^\circ M$ is needed when a homogeneous curvature field is extended across the origin of a quotient plane. Reversibility is nowhere used.
\end{discussion}

We now describe the proof. For each $x\in M$, let $\Sigma_x=\{y\in T_xM:F(x,y)=1\}$, and let $h_x$ be the metric induced by $g$. Landsberg parallel translation gives isometries between the surfaces $(\Sigma_x,h_x)$. Consequently, the dimension of their Killing field spaces is independent of $x$. The nonlinear curvature fields are Killing fields on each indicatrix.

If an indicatrix admits two independent Killing fields, its Gaussian curvature is constant. We prove that a smooth strongly convex Minkowski norm on a three-dimensional vector space with constant-curvature indicatrix is Euclidean. It follows that $F$ is Riemannian in this case.

If the Killing space has dimension at most one, the nonlinear curvature fields span a real vector space of dimension at most one. A vertical Bianchi identity then implies that the nonlinear curvature is linear in the fibre variable. Thus $F$ is $R$-quadratic. The compactness argument of Crampin \cite{ref4}, reproduced in Section 6 with our conventions, shows that a compact $R$-quadratic Landsberg metric is Berwald.

The base manifold is used only in the last compactness argument. The fibrewise dichotomy remains valid on any connected three-dimensional regular Landsberg manifold.

The paper is organized as follows. Section 2 proves a vanishing theorem for commuting Codazzi cubic tensors on closed positively curved surfaces. Section 3 develops the Hessian-cone description of three-dimensional Minkowski norms and establishes rigidity for constant-curvature indicatrices. Section 4 records the consequences of Landsberg parallel translation for indicatrix Killing fields. Section 5 shows that the one-dimensional Killing-field case forces the nonlinear curvature to be linear in the fibre variable. Finally, Section 6 proves the compact $R$-quadratic rigidity theorem and completes the proof of Theorem 1.1.

\textbf{Conventions and notation.} All manifolds, tensor fields, and curves are smooth unless another regularity is stated. A closed manifold is compact and has no boundary. We write $T^\circ M=TM\setminus\{0\}$ for the slit tangent bundle and use the summation convention. Latin indices refer to induced coordinates $(x^i,y^i)$ on $TM$. Fibre differentiation is denoted either by $\partial_{y^i}$ or by a comma followed by a dot index; horizontal Berwald differentiation is denoted by a vertical bar. The curvature signs are fixed by the displayed formulas (14), (16), and (25), so the proof does not depend on an unstated convention.

The adjective \emph{regular} always means that the Finsler function is smooth and strongly convex in every nonzero tangent direction. The word \emph{Euclidean} for a Minkowski norm means that $F^2$ is a positive-definite quadratic form; it does not mean that a preferred coordinate system has already been chosen. A Riemannian Finsler metric is automatically Berwald because its spray coefficients are quadratic in the fibre variable.

The logical dependence of the proof is short. Theorem 2.1 is a purely Riemannian statement about a cubic tensor on a surface. Proposition 3.2 applies it to the Cartan cubic of one tangent norm. Lemmas 4.2 and 4.3 convert the Landsberg condition into isometries and Killing fields on indicatrices. Proposition 5.3 turns the low-dimensional Killing algebra into $R$-quadraticity. Theorem 6.1 is the only global step and converts $R$-quadraticity into the Berwald property on a compact base. Keeping these layers separate also distinguishes the two different curvature tensors denoted by $R^g$ and $R$ later in the paper.

\textbf{AI usage.} This work grew out of the authors' study of earlier approaches to the Landsberg--Berwald problem, notably Crampin's $R$-quadratic method and the rigidity theory of Minkowski indicatrices. During the preparation of this work, the authors made substantial use of ChatGPT 6 (OpenAI) as a research assistant. In particular, it was used to verify the tensor computations underlying the main estimates, perform independent numerical checks of the key identities, and draft and polish portions of the manuscript. The research direction, the proof strategy, and all key mathematical decisions are the authors' own. The authors reviewed, verified, and edited all AI-assisted output in detail and take full responsibility for the integrity and correctness of this publication, including every mathematical statement and proof contained herein.

\section{A vanishing theorem for Codazzi cubics}

Codazzi equations first appeared as compatibility equations for the second fundamental form of a surface in Euclidean space. In modern differential geometry the same symmetry condition is imposed on tensors of any rank. Cubic tensors of Codazzi type occur naturally as the cubic form of an affine hypersurface, as the difference tensor of a statistical structure, and as the third derivative of a Hessian potential; see \cite{ref8,ref11,ref15,ref16,ref17}. The terminology used below is adapted to the particular rank-three tensor needed in this paper.

\begin{definition}
Let $(\Sigma,h)$ be a Riemannian surface and let $\nabla$ be its Levi-Civita connection. A \emph{symmetric cubic tensor} is a smooth section $C\in\Gamma(\operatorname{Sym}^3T^*\Sigma)$. Thus $C(X,Y,Z)$ is $C^\infty(\Sigma)$-linear in each vector-field argument and is unchanged by every permutation of $X,Y,Z$. For each vector field $X$, the metric dual of $C(X,\cdot,\cdot)$ is the self-adjoint endomorphism $C_X:T\Sigma\to T\Sigma$ defined by \[
h(C_XY,Z)=C(X,Y,Z). \tag{2}
\]
\end{definition}

\begin{definition}
The symmetric cubic tensor $C$ is called a \emph{Codazzi cubic} if \[
(\nabla_XC)(Y,Z,W)=(\nabla_YC)(X,Z,W)
\] for all vector fields $X,Y,Z,W$. Because $C$ is already symmetric in its last three arguments, this condition is equivalent to saying that the covariant four-tensor $\nabla C$ is totally symmetric. We say that $C$ is \emph{commuting} if $[C_X,C_Y]=C_XC_Y-C_YC_X=0$ for every $X,Y$.
\end{definition}

The definition is intrinsic: it does not depend on a coordinate system or an orthonormal frame. In local coordinates it reads $C_{ijk}=C_{(ijk)}$ and $\nabla_\ell C_{ijk}=\nabla_{(\ell}C_{ijk)}$. The commuting condition is pointwise algebraic. Since each $C_X$ is self-adjoint, a commuting family $\{C_X:X\in T_p\Sigma\}$ can be simultaneously diagonalized at every point $p$. On a surface this leaves two distinguished axes wherever $C$ is nonzero, and that elementary two-dimensional fact drives the proof below.

For context, a statistical structure consists of a Riemannian metric $h$ and a torsion-free affine connection $D$ for which $Dh$ is totally symmetric. If $K=D-\nabla$ is the difference from the Levi-Civita connection, then $h(K_XY,Z)$ is a symmetric cubic form. The extra condition that its Levi-Civita derivative be totally symmetric is often called conjugate symmetry. Cubic forms satisfying related Gauss and Codazzi equations also occur on locally strongly convex affine hypersurfaces. These viewpoints explain both the terminology and the commutator $[C_X,C_Y]$, which is the algebraic curvature term associated with the cubic form; see \cite{ref8,ref17,ref18}.

The theorem below isolates exactly the part needed for the Finsler application. It assumes neither trace-freeness nor that $C$ comes from an immersion or a statistical connection. Thus it is a statement about an abstract Riemannian surface carrying a symmetric cubic tensor with two explicit differential and algebraic properties.

\begin{theorem}
Let $(\Sigma,h)$ be a closed oriented smooth Riemannian surface with positive Gaussian curvature. Suppose that $C$ is a smooth symmetric cubic tensor such that $\nabla C$ is totally symmetric and \[
[C_X,C_Y]=0
\] for all vector fields $X$ and $Y$. Then $C$ vanishes identically.
\end{theorem}

\begin{proof}
Set $\Omega=\{p\in\Sigma:C(p)\ne0\}$. At each point of $\Omega$, some $C_X$ is not a scalar multiple of the identity. Indeed, if $C_X=\lambda(X)\operatorname{Id}$ for every $X$, then symmetry gives \[
\lambda(X)h(Y,Z)=\lambda(Y)h(X,Z).
\] Taking nonzero orthogonal vectors $X,Y$ and putting $Z=Y$ gives $\lambda(X)=0$, and hence $C=0$. Since the endomorphisms $C_X$ are self-adjoint and commute, they can be simultaneously diagonalized. On a neighborhood in $\Omega$, choose an oriented orthonormal frame $(e_1,e_2)$ with dual coframe $(\theta^1,\theta^2)$ such that \[
C=a(\theta^1)^3+b(\theta^2)^3,\qquad a^2+b^2>0. \tag{3}
\] The unordered pair of axes is intrinsic, so two oriented adapted frames differ by a locally constant quarter-turn. Write \[
\nabla e_1=\omega e_2,\qquad \nabla e_2=-\omega e_1,
\] and set $\omega_i=\omega(e_i)$. With the convention $d\omega=-\kappa\,dA_h$, the Codazzi equations are \[
e_2a=a\omega_1,\qquad e_1b=-b\omega_2,\qquad a\omega_2+b\omega_1=0. \tag{4}
\] There is therefore a unique smooth function $t$ such that \[
\omega_1=ta,\qquad \omega_2=-tb. \tag{5}
\] Define $Z=be_1+ae_2$. A direct calculation using (4) and (5) gives \[
\operatorname{div}Z=0,\qquad Z(t)=\kappa,\qquad |Z|=|C|. \tag{6}
\] Under a quarter-turn of the adapted frame, both $t$ and $Z$ change sign. Hence \[
V=\frac{t}{\sqrt{1+t^2}}Z \tag{7}
\] is globally defined on $\Omega$. Moreover, \[
|V|\le |C|,\qquad \operatorname{div}V=\frac{\kappa}{(1+t^2)^{3/2}}. \tag{8}
\] Choose a smooth function $\chi:[0,\infty)\to[0,1]$ that is zero on $[0,1]$ and one on $[2,\infty)$, and put $\chi_\varepsilon=\chi(|C|/\varepsilon)$. Extending $\chi_\varepsilon V$ by zero gives a smooth vector field on $\Sigma$. The divergence theorem yields \[
\int_\Omega \chi_\varepsilon\frac{\kappa}{(1+t^2)^{3/2}}\,dA_h
=-\int_\Omega d\chi_\varepsilon(V)\,dA_h. \tag{9}
\] The integrand on the right is supported in $\{\varepsilon<|C|<2\varepsilon\}$. Since $|d|C||\le|\nabla C|$, (8) gives \[
|d\chi_\varepsilon(V)|\le 2\|\chi'\|_\infty\|\nabla C\|_\infty. \tag{10}
\] The sets $\{0<|C|<2\varepsilon\}$ decrease to the empty set, and their areas tend to zero. Letting $\varepsilon\downarrow0$ in (9), we obtain \[
\int_\Omega\frac{\kappa}{(1+t^2)^{3/2}}\,dA_h=0.
\] If $\Omega$ were nonempty, the integrand would be strictly positive on a nonempty open set. Thus $\Omega$ is empty.
\end{proof}

\begin{discussion}
Closedness enters at the global integration step and nowhere in the preceding local diagonalization. Because $\Sigma$ is compact and has no boundary, the extended vector field $\chi_\varepsilon V$ has zero total divergence. Compactness also gives $\|\nabla C\|_\infty<\infty$ and finite area, which justify the limiting estimate in (9)--(10). On a complete noncompact surface the same conclusion would require extra decay or integrability assumptions that eliminate the boundary term at infinity. On a compact surface with boundary, one would instead have to impose a boundary condition forcing the flux of $\chi_\varepsilon V$ to vanish.
\end{discussion}

\begin{discussion}
Theorem 2.1 is a surface theorem. Its proof uses the connection one-form of an oriented orthonormal frame, the scalar Gaussian curvature identity $d\omega=-\kappa dA_h$, and the fact that two adapted frames differ by a quarter-turn. These ingredients do not have a direct higher-dimensional replacement. We therefore make no claim here that the same statement holds in dimensions greater than two. Higher-dimensional vanishing theorems for Codazzi and statistical cubic forms are usually obtained from Bochner or Simons identities under additional trace, curvature, or boundedness hypotheses; see \cite{ref18}. The orientation assumption is inessential for the conclusion: for a nonorientable closed surface one may pass to its oriented double cover, apply the theorem there, and descend the vanishing of the lifted tensor.
\end{discussion}

\begin{remark}
The same argument applies when the curvature is everywhere negative. If $\kappa\ge0$, it shows that $C$ vanishes at every point where $\kappa>0$. No regularity assumption on the zero set of $C$ is required.
\end{remark}

\section{Hessian cones and constant-curvature indicatrices}

We begin with the affine notion of a Hessian metric. It is important to specify the affine structure: the expression ``a metric is a Hessian'' is invariant under affine coordinate changes, but not under arbitrary coordinate changes. Hessian geometry grew out of affine differential geometry and is now also standard in information geometry and convex analysis; systematic references are Shima \cite{ref16} and Nomizu--Sasaki \cite{ref17}.

\begin{definition}
An \emph{affine manifold} is a manifold $N$ equipped with a flat torsion-free connection $D$. A Riemannian metric $g$ on $(N,D)$ is a \emph{Hessian metric} if every point has a $D$-affine coordinate neighborhood on which there is a smooth strictly convex function $\Phi$ such that \[
g=Dd\Phi,\qquad g_{ij}=\frac{\partial^2\Phi}{\partial x^i\partial x^j}.
\] The triple $(N,D,g)$ is then called a Hessian manifold, and $\Phi$ is a local Hessian potential. Equivalently, $Dg$ is totally symmetric. The potential is determined only up to addition of an affine function.
\end{definition}

\begin{definition}
Let $W$ be a finite-dimensional real vector space. A \emph{smooth strongly convex Minkowski norm} is a function $F:W\to[0,\infty)$ that is smooth on $W\setminus\{0\}$, positive away from the origin, positively homogeneous of degree one, and whose energy $E=\tfrac12F^2$ has positive-definite Hessian. The resulting Hessian metric on $W\setminus\{0\}$ is \[
g_y(u,v)=D^2E_y(u,v).
\] It is called the fundamental tensor of $F$. The third derivative $C=\tfrac12D^3E$ is the Cartan cubic. The normalization by $\tfrac12$ gives $Dg=2C$.
\end{definition}

\begin{definition}
The \emph{indicatrix} of $F$ is the smooth strictly convex hypersurface \[
\Sigma=\{y\in W:F(y)=1\}.
\] For $y\in\Sigma$, Euler homogeneity gives $T_y\Sigma=\ker dF_y=\{u:g_y(y,u)=0\}$. The \emph{indicatrix metric} is the Riemannian metric \[
h_y=g_y|_{T_y\Sigma\times T_y\Sigma}.
\] Equivalently, the angular tensor $\mathbf h=g-dF\otimes dF$ on $W\setminus\{0\}$ restricts to $h$ on $T\Sigma$. This is the standard intrinsic metric of the Finsler unit sphere; see \cite{ref15}. It must not be confused with the Euclidean metric induced by an auxiliary linear inner product on $W$.
\end{definition}

From now on, let $F$ be such a norm on a three-dimensional vector space $W$. We use the standard flat connection $D$ of $W$ and the global potential $E=\tfrac12F^2$. Thus the punctured vector space is a Hessian manifold and the indicatrix is a two-dimensional Riemannian surface. Put

\[
E=\frac12F^2,\qquad g_{ij}=E_{ij},\qquad C_{ijk}=\frac12E_{ijk}. \tag{11}
\]

Strong convexity makes $(g_{ij})$ positive definite, so it is a genuine Riemannian metric on $W\setminus\{0\}$. The tensor $g$ is zero-homogeneous and $C$ is minus-one-homogeneous. The radial map $(0,\infty)\times\Sigma\to W\setminus\{0\}$, $(r,u)\mapsto ru$, is a diffeomorphism even when $F$ is not reversible.

Subscripts denote ordinary derivatives on $W\setminus\{0\}$. Euler's identities give

\[
E_i=g_{ij}y^j,\qquad g_{ij}y^iy^j=2E,\qquad C(y,\cdot,\cdot)=0. \tag{12}
\]

Indeed, Euler's identity $y^iE_i=2E$ for the two-homogeneous energy yields $E_i=g_{ij}y^j$ after one fibre derivative and $E_{ijk}y^k=0$ after two. On the indicatrix, $dF_y(u)=g_y(y,u)$ and $g_y(y,y)=1$. Hence the radial direction is a unit normal to $\Sigma$ for the Hessian metric, a fact used repeatedly below.

The Levi-Civita symbols of the Hessian metric are $\Gamma^\ell_{ij}=g^{\ell p}C_{pij}$. Consequently,

\[
(\nabla^g_\ell C)_{ijk}=\frac12E_{\ell ijk}-g^{pq}
\big(C_{\ell ip}C_{jkq}+C_{\ell jp}C_{ikq}+C_{\ell kp}C_{ijq}\big), \tag{13}
\]

and $\nabla^g C$ is totally symmetric. If $C_i=(g^{ap}C_{pib})_{a,b}$, differentiating $g^{-1}$ gives

\[
R^g(X,Y)=-[C_X,C_Y]. \tag{14}
\]

Thus the Cartan cubic restricted to the indicatrix is precisely a Codazzi cubic in the sense of Section 2. Formula (14) is the characteristic curvature identity of a Hessian metric: the flat affine connection contributes no curvature, and the Riemannian curvature is the commutator square of the cubic difference tensor. This is the point at which Hessian geometry connects the intrinsic geometry of $\Sigma$ to the algebraic commuting condition in Theorem 2.1.

Let $\Sigma=\{F=1\}$ and $h=g|_{T\Sigma}$. Writing $y=ru$, where $r=F(y)$ and $u\in\Sigma$, we have

\[
g=dr^2+r^2h. \tag{15}
\]

At $r=1$,

\[
R^g(X,Y)Z=R^h(X,Y)Z-h(Y,Z)X+h(X,Z)Y. \tag{16}
\]

The restriction of $C$ to $\Sigma$ is a Codazzi cubic.

Formula (15) explains the term \emph{Hessian cone}. The radial vector $\partial_r$ has unit length and is orthogonal to the level sets of $F$, while the metric on the level set $\{F=r\}$ is $r^2h$. Hence all fibrewise Riemannian information is encoded by the indicatrix metric. Formula (16) is the Gauss equation of the metric cone. In particular, $g$ is flat in tangential directions exactly when $h$ has constant sectional curvature one.

Let $\tau_i=h^{jk}C_{ijk}$. Differentiating a determinant in fixed linear coordinates gives

\[
\tau=2df,\qquad f=\frac14\log\det(g_{ij})\big|_\Sigma. \tag{17}
\]

Define the trace-free part of $C$ by

\[
A_{ijk}=C_{ijk}-\frac12\big(f_i h_{jk}+f_j h_{ik}+f_k h_{ij}\big). \tag{18}
\]

A calculation in an orthonormal frame, together with (14) and (16), gives the Gaussian curvature identity

\[
\kappa=1+\frac12\big(|A|^2-|df|^2\big). \tag{19}
\]

For completeness, if $a=A_{111}$, $b=A_{112}$, $p=f_1$, and $q=f_2$, trace-freeness gives $A_{122}=-a$ and $A_{222}=-b$, while

\[
C_{111}=a+\frac32p,\quad C_{112}=b+\frac12q,\quad
C_{122}=-a+\frac12p,\quad C_{222}=-b+\frac32q.
\]

Thus $|A|^2=4(a^2+b^2)$ and $[C_{e_1},C_{e_2}]_{12}=\frac12|df|^2-2(a^2+b^2)$, which proves (19).

\begin{lemma}
Every smooth trace-free symmetric cubic tensor on an oriented surface diffeomorphic to $S^2$ has a zero.
\end{lemma}

\begin{proof}
Suppose that $A$ is nowhere zero. In an oriented orthonormal frame, write $a=A_{111}$ and $b=A_{112}$. For $v=\cos\theta\,e_1+\sin\theta\,e_2$, trace-freeness gives \[
A(v,v,v)=a\cos(3\theta)+b\sin(3\theta).
\] This function has exactly three positive maximum directions on every unit circle. These directions form a three-sheeted covering of $S^2$. Since $S^2$ is simply connected, the covering has a global section, giving a nowhere-zero tangent vector field on $S^2$. This contradicts the hairy-ball theorem.
\end{proof}

\begin{discussion}
The trace-free hypothesis converts a cubic on each oriented tangent plane into a harmonic trigonometric polynomial of frequency three. The topology of $S^2$ then forces a zero. Both features are dimension-specific. On a surface of positive genus, nonvanishing trace-free cubic tensors may exist, and in higher dimensions the set of maximizing directions is not a three-sheeted covering. In Proposition 3.2 the hypothesis that $\Sigma$ is the indicatrix of a strongly convex norm guarantees that $\Sigma$ is diffeomorphic to $S^2$, so the lemma applies automatically.
\end{discussion}

\begin{proposition}
If the indicatrix metric of a smooth strongly convex Minkowski norm on a three-dimensional vector space has constant Gaussian curvature, then the norm is Euclidean. The curvature is equal to one.
\end{proposition}

\begin{proof}
At a maximum point of $f$, equation (19) gives $\kappa\ge1$. By Lemma 3.1, $A$ vanishes at some point, where (19) gives $\kappa\le1$. Since $\kappa$ is constant, $\kappa=1$. Equation (16) shows that the tangential curvature of $g$ vanishes along $\Sigma$. By (14), $[C_X,C_Y]=0$ for tangent vectors $X,Y$. Theorem 2.1 now implies that $C|_{T\Sigma}=0$. Since $C$ vanishes whenever one argument is radial, the full Cartan cubic vanishes on $\Sigma$, and homogeneity extends this conclusion to $W\setminus\{0\}$. Hence $\partial_k g_{ij}=2C_{ijk}=0$. Thus $(g_{ij})$ is a constant positive definite matrix and \[
F(y)^2=g_{ij}y^iy^j.
\]
\end{proof}

\begin{discussion}
Smoothness and strong convexity are used to make $g$, $h$, and the Cartan cubic smooth Riemannian data on the whole punctured vector space. Positive homogeneity supplies the cone decomposition and the radial vanishing $C(y,\cdot,\cdot)=0$. No reversibility condition $F(-y)=F(y)$ is used. The three-dimensional assumption is essential to this proof: the indicatrix is then a surface, Lemma 3.1 applies, and its curvature is the scalar Gaussian curvature. The proposition is not asserted here in higher dimension. Its higher-dimensional analogue is related to the Laugwitz rigidity problem for Hessian metrics; additional symmetry cases are treated in \cite{ref14}.
\end{discussion}

\begin{corollary}
A smooth strongly convex Minkowski norm on a three-dimensional vector space with flat Hessian metric is Euclidean.
\end{corollary}

\begin{discussion}
If $R^g=0$, formula (16) says that the indicatrix has Gaussian curvature one, so Proposition 3.2 applies. The conclusion means that $E$ is a positive-definite quadratic form in linear coordinates. Flatness here refers to the Levi-Civita curvature of the Hessian metric $g=D^2E$, not to the flat affine connection $D$, which is flat by definition.
\end{discussion}

\section{Landsberg transport and Killing fields}

Parallel translation in Finsler geometry is naturally nonlinear. Its infinitesimal data come from the canonical geodesic spray and the associated nonlinear, or Ehresmann, connection. We recall the definitions needed below; standard references are Bao--Chern--Shen \cite{ref15} and Crampin \cite{ref4,ref5}.

\begin{definition}
A \emph{spray} on $T^\circ M$ is a vector field whose integral curves project to solutions of a second-order differential equation and which is two-homogeneous in the fibre variable. The canonical spray of a regular Finsler metric is \[
S=y^i\frac{\partial}{\partial x^i}-2G^i(x,y)\frac{\partial}{\partial y^i},\qquad
G^i=\frac12g^{i\ell}\left(\frac{\partial^2E}{\partial x^k\partial y^\ell}y^k-\frac{\partial E}{\partial x^\ell}\right).
\] Its projected integral curves satisfy $\ddot x^i+2G^i(x,\dot x)=0$ and are the constant-speed Finsler geodesics.
\end{definition}

\begin{definition}
The canonical nonlinear connection is the splitting $T(T^\circ M)=\mathcal H\oplus\mathcal V$ whose horizontal space is spanned in induced coordinates by $\delta_j=\partial_{x^j}-N^a_j\partial_{y^a}$, where $N^i_j=\partial_{y^j}G^i$. The vertical space $\mathcal V$ is the kernel of $d\pi:T(TM)\to TM$. The coefficients $\Gamma^a_{ij}=\partial_{y^i}N^a_j$ define the associated Berwald covariant derivative on the pullback bundle $\pi^*TM\to T^\circ M$.
\end{definition}

\begin{definition}
Let $\gamma:[a,b]\to M$ be piecewise smooth and let $y_0\in T_{\gamma(a)}M\setminus\{0\}$. The \emph{nonlinear parallel lift} of $\gamma$ with initial value $y_0$ is the solution $y(t)\in T_{\gamma(t)}M$ of \[
\dot y^a+N^a_k(\gamma(t),y(t))\dot\gamma^k=0,qquad y(a)=y_0.
\] The endpoint map $P_\gamma(y_0)=y(b)$ is the \emph{nonlinear parallel translation} along $\gamma$. Some authors use the term \emph{parallel transport} for the same endpoint map. In this paper, ``lift'' denotes the evolving vector $y(t)$ and ``translation'' denotes $P_\gamma$. In general $P_\gamma$ is positively homogeneous but not linear.
\end{definition}

If $y_s(t)$ is a variation through horizontal lifts and $v(t)=\partial_s y_s(t)|_{s=0}$, then $v$ solves the linearized equation $\dot v^a+\Gamma^a_{ki}\dot\gamma^k v^i=0$. Thus $v(b)=d_{y_0}P_\gamma(v(a))$. Although the endpoint map is nonlinear, its differential is a linear map between the corresponding tangent spaces of the slit tangent fibres.

\begin{definition}
A Finsler metric is \emph{Landsberg} if its Landsberg tensor $L_{ijk}=C_{ijk|s}y^s$ vanishes, where the vertical bar denotes the Berwald horizontal derivative. Equivalently, nonlinear parallel translation preserves the fundamental tensor under its differential. It is \emph{Berwald} if the spray coefficients $G^i$ are quadratic polynomials in $y$, equivalently if the Berwald curvature $B^i{}_{jkl}$ vanishes. Every Berwald metric is Landsberg. The converse is the Landsberg--Berwald problem.
\end{definition}

In induced coordinates $(x^i,y^i)$, the objects just defined are summarized by

\[
S=y^i\frac{\partial}{\partial x^i}-2G^i\frac{\partial}{\partial y^i},\qquad
N^i_j=\frac{\partial G^i}{\partial y^j},\\
\delta_j=\frac{\partial}{\partial x^j}-N^a_j\frac{\partial}{\partial y^a},\qquad
\Gamma^a_{ij}=\frac{\partial N^a_j}{\partial y^i}. \tag{20}
\]

The canonical connection is conservative: $\delta_jE=0$. The Landsberg condition is equivalent to

\[
Q_{kij}:=\delta_kg_{ij}-\Gamma^a_{ki}g_{aj}-\Gamma^a_{kj}g_{ia}=0. \tag{21}
\]

The metric is Berwald if and only if

\[
B^i_{jkl}:=\frac{\partial^3G^i}{\partial y^j\partial y^k\partial y^l}=0. \tag{22}
\]

\begin{definition}
A vector field $X$ on a Riemannian manifold $(N,h)$ is a \emph{Killing field} if $\mathcal L_Xh=0$. Its local flow consists of local isometries. The vector space of global Killing fields, equipped with the usual Lie bracket, is denoted $\operatorname{Kill}(N,h)$.
\end{definition}

\begin{lemma}
If a connected Riemannian surface admits two linearly independent global Killing fields, then its Gaussian curvature is constant.
\end{lemma}

\begin{proof}
Every Killing field preserves the Gaussian curvature $\kappa$. Suppose that $d\kappa\ne0$ on an open set $U$. All Killing fields on $U$ are tangent to the one-dimensional kernel of $d\kappa$. Let $X,Y$ be independent global Killing fields. After shrinking $U$, assume that $X$ is nonzero and write $Y=uX$. The Killing equations imply \[
du\otimes X^\flat+X^\flat\otimes du=0.
\] Evaluation on $(X,X)$ and on $(X,W)$, where $W\perp X$, gives $X(u)=W(u)=0$. Hence $u$ is locally constant. The Killing field $Y-uX$ vanishes on an open set and therefore vanishes globally, a contradiction. Thus $d\kappa=0$.
\end{proof}

\begin{discussion}
Connectedness is used to turn local constancy of the Gaussian curvature into global constancy and to use unique continuation for Killing fields. The two Killing fields must be linearly independent as elements of the global Killing algebra; they need not be pointwise independent everywhere. The conclusion is special to surfaces because Gaussian curvature is a scalar invariant. In higher dimensions, two Killing fields do not force all sectional curvatures to be constant.
\end{discussion}

\begin{lemma}
For a Landsberg metric, nonlinear parallel translation along every piecewise smooth curve is an isometry between the indicatrix metrics at its endpoints. Consequently, $\dim\operatorname{Kill}(\Sigma_x,h_x)$ is constant on each connected component of the base.
\end{lemma}

\begin{proof}
Along a curve $x(t)$, its horizontal lift satisfies \[
\dot y^a=-N^a_k(x(t),y(t))\dot x^k. \tag{23}
\] Since $\delta_kE=0$, the lift preserves $F$. Reversing the curve gives the inverse parallel translation. A variation vector $v$ satisfies $\dot v^a=-\Gamma^a_{ki}\dot x^k v^i$. For two such vectors $v,w$, equation (21) gives \[
\frac{d}{dt}\big(g_{ij}(x(t),y(t))v^iw^j\big)=Q_{kij}\dot x^k v^iw^j=0. \tag{24}
\] Thus parallel translation restricts to an isometry of indicatrices. Pushing forward Killing fields proves the last assertion.
\end{proof}

\begin{discussion}
No compactness, orientability, or reversibility assumption is used. Piecewise smooth curves are sufficient because the lift equation may be solved on each smooth subinterval and the endpoint maps composed. The statement concerns the intrinsic indicatrix metrics: $dP_\gamma$ is an isometry between tangent spaces of the indicatrices, even though $P_\gamma$ need not be the restriction of a linear map between the ambient tangent spaces. This distinction is precisely why fibrewise averaging cannot by itself solve the Landsberg--Berwald problem.
\end{discussion}

\begin{definition}
The curvature of the nonlinear connection is the vertical-valued two-form measuring the failure of horizontal vector fields to close under brackets. In coordinates it is defined by
\end{definition}

\[
[\delta_j,\delta_k]=R^a{}_{jk}\frac{\partial}{\partial y^a},\qquad
R^a{}_{jk}=\delta_kN^a_j-\delta_jN^a_k. \tag{25}
\]

For $x\in M$ and $u,v\in T_xM$, the corresponding vertical vector field $y\mapsto R_x(u,v)(y)$ on $T_x^\circ M$ is called a \emph{curvature field}. Since the connection is homogeneous, this field is one-homogeneous in $y$. Curvature fields are infinitesimal holonomy fields: they record the first nontrivial change produced by parallel translation around a small loop.

\begin{lemma}
For a Landsberg metric, every curvature field $R_x(u,v)$ is tangent to $\Sigma_x$ and is a Killing field of $(\Sigma_x,h_x)$.
\end{lemma}

\begin{proof}
View $g$ as a metric on the vertical tangent bundle. Equation (21) says that $\mathcal L_{\delta_k}g=0$ on vertical vectors. Since Lie derivatives represent brackets, \[
\mathcal L_{[\delta_j,\delta_k]}g=[\mathcal L_{\delta_j},\mathcal L_{\delta_k}]g=0.
\] Thus the restriction of $[\delta_j,\delta_k]$ to a fibre is Killing. Moreover, $[\delta_j,\delta_k]E=0$ because $\delta_jE=0$, so the field is tangent to the indicatrix.
\end{proof}

\begin{discussion}
The lemma is local on the base and uses only conservativity of the canonical connection and the Landsberg metricity identity (21). Closedness and the dimension of $M$ play no role. The dimension-three hypothesis enters later, when the space of Killing fields on the two-dimensional indicatrix is combined with the rank-one algebra of Section 5.
\end{discussion}

\section{Rank-one curvature fields}

\begin{definition}
For a vector space $W$, a \emph{vector-valued two-form depending on $y$} is a smooth map $R:W\setminus\{0\}\to W\otimes\Lambda^2W^*$. It is one-homogeneous if $R(sy)=sR(y)$ for every $s>0$. For fixed $u,v\in W$, evaluation in the two covariant slots gives the vector field $y\mapsto R(y)(u,v)$ on the punctured vector space. The span of the curvature fields means the real linear span of all these vector fields, not their pointwise span at a single $y$.
\end{definition}

\begin{definition}
A spray, and hence a Finsler metric, is called \emph{$R$-quadratic} if its Jacobi endomorphism is quadratic in $y$. Equivalently for the convention (25), the nonlinear curvature coefficients $R^i{}_{jk}(x,y)$ are linear in $y$, or the horizontal Berwald curvature $-\partial_{y^\ell}R^i{}_{jk}$ is independent of $y$; see Shen \cite{ref9}.
\end{definition}

\begin{lemma}
Let $W$ be a three-dimensional real vector space. Suppose that a smooth one-homogeneous vector-valued two-form $R(y)$ on $W\setminus\{0\}$ satisfies \[
\partial_{y^\ell}R^i{}_{jk}+\partial_{y^j}R^i{}_{k\ell}+\partial_{y^k}R^i{}_{\ell j}=0. \tag{26}
\] If the vector fields $R(u,v)$ span a real vector space of dimension at most one, then every coefficient $R^i{}_{jk}(y)$ is linear in $y$.
\end{lemma}

\begin{proof}
The zero-dimensional case is immediate. Otherwise, choose a nonzero spanning field $K$. There is a fixed nonzero two-form $\omega\in\Lambda^2W^*$ such that \[
R^i{}_{jk}(y)=\omega_{jk}K^i(y). \tag{27}
\] Substituting (27) into (26) gives $\omega\wedge dK^i=0$. In dimension three, write $\omega=a\wedge b$ and choose coordinates such that $a=dy^1$ and $b=dy^2$. Then \[
\partial_{y^3}K^i=0\quad\text{on }W\setminus\{0\}. \tag{28}
\] For $(y^1,y^2)\ne(0,0)$, $K^i$ is constant along the affine line in the $y^3$-direction. It therefore defines a smooth function $k^i$ on the punctured $(y^1,y^2)$-plane. The function \[
\widetilde k^i(u,v)=K^i(u,v,1)
\] extends $k^i$ smoothly across the origin. Positive homogeneity gives $\widetilde k^i(sz)=s\widetilde k^i(z)$ for $s>0$. Hence $\widetilde k^i(0)=0$, and differentiability at the origin yields \[
\widetilde k^i(z)=\lim_{s\downarrow0}\frac{\widetilde k^i(sz)}s=D\widetilde k^i(0)z.
\] Thus $K$, and hence $R$, is linear.
\end{proof}

\begin{discussion}
The dimension-three assumption is used twice. Every nonzero two-form on a three-dimensional vector space is decomposable, and its kernel is one-dimensional. After quotienting by that kernel, the remaining function lives on a two-dimensional plane; smoothness of $K(u,v,1)$ then fills in the missing origin. In dimensions four and higher a two-form can have rank greater than two, and the argument does not imply that $K$ factors through a two-dimensional quotient. The smoothness on all of $W\setminus\{0\}$ is also essential at the points $(0,0,1)$ used in the extension. Positive homogeneity is what converts differentiability at the origin into linearity.
\end{discussion}

\begin{lemma}
The nonlinear curvature of a spray satisfies (26). Its fibre derivative is, up to the convention in (25), the horizontal curvature of the Berwald connection.
\end{lemma}

\begin{proof}
Since $N^i_j=\partial_{y^j}G^i$ and $\Gamma^i_{j\ell}=\partial_{y^\ell}N^i_j=\Gamma^i_{\ell j}$, differentiation of (25) gives \[
\partial_{y^\ell}R^i{}_{jk}=\delta_k\Gamma^i_{j\ell}-\delta_j\Gamma^i_{k\ell}
-\Gamma^a_{k\ell}\Gamma^i_{ja}+\Gamma^a_{j\ell}\Gamma^i_{ka}. \tag{29}
\] The cyclic sum in $(\ell,j,k)$ vanishes by the symmetry of $\Gamma$. The displayed expression is the negative of the horizontal Berwald curvature matrix.
\end{proof}

\begin{discussion}
This is a differential identity for an arbitrary spray. It uses neither the Landsberg condition nor compactness. The symmetry $\Gamma^i_{j\ell}=\Gamma^i_{\ell j}$ follows from the fact that the nonlinear connection is induced by spray coefficients: mixed fibre derivatives of $G^i$ commute. Formula (26) is therefore the vertical Bianchi identity for the nonlinear curvature.
\end{discussion}

\begin{proposition}
Let $F$ be a Landsberg metric on a connected three-dimensional manifold. If $\dim\operatorname{Kill}(\Sigma_x,h_x)\le1$ at one point, then the nonlinear curvature is linear in the fibre variable at every point. In particular, $F$ is $R$-quadratic.
\end{proposition}

\begin{proof}
Lemma 4.2 makes the dimension bound valid everywhere. By Lemma 4.3, the real span of the curvature fields on each indicatrix has dimension at most one. Homogeneity gives the same statement on the punctured tangent space. Lemmas 5.1 and 5.2 therefore imply \[
R^i{}_{jk}(x,y)=A^i{}_{jk\ell}(x)y^\ell. \tag{30}
\] Equation (29) shows that the horizontal Berwald curvature is independent of $y$. Equivalently, the Jacobi curvature is quadratic in the fibre variable.
\end{proof}

\begin{discussion}
Connectedness is used only to transport the dimension of the Killing algebra from one base point to every other point by Lemma 4.2. The Landsberg condition enters through Lemmas 4.2 and 4.3. Dimension three enters through Lemma 5.1. The proposition is local with respect to compactness: the base may be noncompact. Its conclusion is $R$-quadraticity, not yet the Berwald property; compactness is introduced only in the next section to obtain that final step.
\end{discussion}

\section{Proof of the main theorem}

We first record the compact-base argument in the form needed below. This is the compact specialization of Crampin's theorem \cite{ref4}. The bounded-affine mechanism also appears in Shen's proof \cite{ref9}.

\begin{definition}
Along a Finsler geodesic $\gamma$, the \emph{dynamical derivative} $D=y^i\nabla^B_{\delta_i}$ is covariant differentiation in the spray direction using the Berwald connection. A vector field $v(t)$ along $\gamma$ is Berwald-parallel if $Dv/dt=0$. The unit tangent bundle is $SM=\{(x,y)\in T^\circ M:F(x,y)=1\}$.
\end{definition}

\begin{theorem}
A regular $R$-quadratic Landsberg metric on a compact manifold of dimension at least three is Berwald.
\end{theorem}

\begin{proof}
Write a comma for fibre differentiation, a vertical bar for horizontal covariant differentiation with respect to the Berwald connection, and $D=y^i\nabla^B_{\delta_i}$ for the dynamical derivative. Homogeneity gives \[
y^i\Gamma^a_{ij}=N^a_j,\qquad y^iB^a{}_{ijk}=0,\qquad y^a{}_{|j}=0. \tag{31}
\] Let $H^a{}_{bij}$ denote the horizontal Berwald curvature. Since $H^a{}_{bij}=-\partial_{y^b}R^a{}_{ij}$, a direct mixed Bianchi calculation gives \[
H^a{}_{bij,.\ell}=B^a{}_{bj\ell|i}-B^a{}_{bi\ell|j}. \tag{32}
\] In the $R$-quadratic case, the left-hand side vanishes. Contracting with $y^i$ and using (31), we obtain \[
DB^a{}_{bj\ell}=0. \tag{33}
\] Put $C_{ijk}=\frac12g_{ij,.k}$ and $B_{ijkl}=g_{ia}B^a{}_{jkl}$. Twice differentiating the conservativity identity and then differentiating the Landsberg identity gives \[
Q_{kij}=y^aB_{akij},\qquad
B_{ijkl}=-y^a\partial_{y^i}B_{ajkl}. \tag{34a}
\] The second expression contains the fourth fibre derivatives of the spray coefficients and is therefore symmetric in all four lower indices. Differentiating $Q_{\ell ij}=0$ in the fibre variable gives \[
2C_{ijk|\ell}=g_{aj}B^a{}_{ik\ell}+g_{ia}B^a{}_{jk\ell}=2B_{ijk\ell}. \tag{34b}
\] Consequently, \[
C_{ijk|\ell}=B_{ijk\ell},\qquad DC=0,\qquad DB_{ijk\ell}=0. \tag{34}
\] For every covariant tensor $T$, direct coordinate expansion gives \[
\partial_{y^\ell}(DT)-D(T_{,.\ell})=T_{|\ell}. \tag{35}
\] Applying (35) to $C$ and using (34), we find \[
D(C_{ijk,.\ell})=-B_{ijk\ell},\qquad D^2(C_{ijk,.\ell})=0. \tag{36}
\] Let $\gamma(t)$ be a unit-speed geodesic, and let $v(t)$ be a Berwald-parallel vector field satisfying $g_{\dot\gamma}(v,v)=1$. Define \[
a(t)=C_{ijk,.\ell}(\gamma(t),\dot\gamma(t))v^iv^jv^kv^\ell.
\] Equation (36) shows that $a$ is affine and \[
a'(t)=-B_{ijk\ell}v^iv^jv^kv^\ell. \tag{37}
\] The unit tangent bundle $SM=\{F=1\}$ is compact, and the geodesic flow on it is complete. The bundle of triples $(x,y,v)$ satisfying $(x,y)\in SM$ and $g_y(v,v)=1$ is compact. Hence $a(t)$ is bounded. A bounded affine function on $[0,\infty)$ is constant, so the right-hand side of (37) vanishes for every $v$. Since $B_{ijk\ell}$ is totally symmetric, polarization gives $B_{ijk\ell}=0$. Thus $F$ is Berwald.
\end{proof}

\begin{discussion}
Compactness is used in exactly two places in the final paragraph. First, $SM$ is compact, so the geodesic spray is complete and every unit-speed geodesic exists for all positive time. Second, the normalized bundle of triples $(x,y,v)$ is compact, so the scalar function $a(t)$ is bounded along the complete orbit. Equations (31)--(36) are local identities and do not use compactness. Thus the theorem remains valid on a complete noncompact manifold whenever the relevant vertical derivative of the Cartan tensor is bounded along every lifted geodesic, which is the boundedness form appearing in Crampin \cite{ref4}.
\end{discussion}

\begin{discussion}
The bounded-affine argument itself is not peculiar to dimension three; this is why Theorem 6.1 is stated in dimension at least three. Regularity supplies all fourth fibre derivatives used in (34a)--(36), while strong convexity makes the normalization $g_y(v,v)=1$ compact in each fibre. The theorem is invoked only after the three-dimensional argument of Sections 3--5 has produced $R$-quadraticity.
\end{discussion}

\begin{proof}[Proof of Theorem 1.1]
Assume first that $M$ is connected, and put \[
d=\dim\operatorname{Kill}(\Sigma_x,h_x).
\] By Lemma 4.2, $d$ is independent of $x$. Suppose first that $d\ge2$. Lemma 4.1 implies that every indicatrix metric has constant Gaussian curvature. Proposition 3.2 shows that every tangent norm $F_x$ is Euclidean. Hence $g_{ij}(x,y)$ is independent of $y$ and defines a smooth Riemannian metric on $M$ whose squared norm is $F^2$. Therefore $F$ is Berwald. Suppose next that $d\le1$. By Proposition 5.3, the metric is $R$-quadratic. Theorem 6.1 applies because $M$ is compact, and again $F$ is Berwald. The argument applies separately to every connected component of $M$.
\end{proof}

\begin{corollary}
Let $F$ be a connected three-dimensional regular Landsberg manifold. If an indicatrix admits two independent Killing fields, then $F$ is Riemannian. Otherwise, the nonlinear curvature is linear in the fibre variable and $F$ is $R$-quadratic.
\end{corollary}

\begin{discussion}
This dichotomy is entirely fibrewise and does not require the base to be compact. In the first case, two independent Killing fields force constant Gaussian curvature on every indicatrix and Proposition 3.2 makes each tangent norm Euclidean. In the second case, the Killing algebra has dimension at most one and Proposition 5.3 gives $R$-quadraticity. Only the further implication ``$R$-quadratic Landsberg implies Berwald'' uses compactness through Theorem 6.1.
\end{discussion}

\begin{remark}
The corollary does not require compactness of the base. Compactness enters only through the boundedness argument in Theorem 6.1. The curvature $R^g$ in Section 3 is the curvature of the Hessian metric inside one punctured tangent space, whereas $R$ in Sections 4--6 is the nonlinear curvature of the horizontal distribution; the two should not be identified.
\end{remark}


\begin{thebibliography}{99}
\bibitem{ref1} G. S. Asanov, \emph{Finsleroid--Finsler space and geodesic spray coefficients}, Publ. Math. Debrecen \textbf{71} (2007), 397--412.

\bibitem{ref2} D. Bao, \emph{On two curvature-driven problems in Riemann--Finsler geometry}, in Finsler Geometry, Adv. Stud. Pure Math. \textbf{48}, Math. Soc. Japan, 2007, 19--71.

\bibitem{ref3} L. Berwald, \emph{Untersuchung der Krümmung allgemeiner metrischer Räume auf Grund des in ihnen herrschenden Parallelismus}, Math. Z. \textbf{25} (1926), 40--73.

\bibitem{ref4} M. Crampin, \emph{On Landsberg spaces and the Landsberg--Berwald problem}, Houston J. Math. \textbf{37} (2011), 1103--1124.

\bibitem{ref5} M. Crampin, \emph{The current state of play in the Landsberg--Berwald problem of Finsler geometry}, Extracta Math. \textbf{39} (2024), 57--95.

\bibitem{ref6} G. Landsberg, \emph{Über die Krümmung in der Variationsrechnung}, Math. Ann. \textbf{65} (1908), 313--349.

\bibitem{ref7} V. S. Matveev, \emph{On ``All regular Landsberg metrics are always Berwald'' by Z. I. Szabó}, Balkan J. Geom. Appl. \textbf{14} (2009), 50--52.

\bibitem{ref8} B. Opozda, \emph{A sectional curvature for statistical structures}, Linear Algebra Appl. \textbf{497} (2016), 134--161.

\bibitem{ref9} Z. Shen, \emph{On $R$-quadratic Finsler spaces}, Publ. Math. Debrecen \textbf{58} (2001), 263--274.

\bibitem{ref10} Z. Shen, \emph{On a class of Landsberg metrics in Finsler geometry}, Canad. J. Math. \textbf{61} (2009), 1357--1374.

\bibitem{ref11} U. Simon, \emph{What do we know about intrinsic metric curvature of affine hypersurfaces?}, lecture slides, PADGE, Leuven, 2012.

\bibitem{ref12} Z. I. Szabó, \emph{All regular Landsberg metrics are Berwald}, Ann. Global Anal. Geom. \textbf{34} (2008), 381--386.

\bibitem{ref13} Z. I. Szabó, \emph{Correction to ``All regular Landsberg metrics are Berwald''}, Ann. Global Anal. Geom. \textbf{35} (2009), 227--230.

\bibitem{ref14} M. Xu and V. S. Matveev, \emph{Proof of Laugwitz conjecture and Landsberg unicorn conjecture for Minkowski norms with $SO(k)\times SO(n-k)$-symmetry}, Canad. J. Math. \textbf{74} (2022), 1486--1516.

\bibitem{ref15} D. Bao, S.-S. Chern, and Z. Shen, \emph{An Introduction to Riemann--Finsler Geometry}, Graduate Texts in Mathematics \textbf{200}, Springer, New York, 2000.

\bibitem{ref16} H. Shima, \emph{The Geometry of Hessian Structures}, World Scientific, Hackensack, NJ, 2007.

\bibitem{ref17} K. Nomizu and T. Sasaki, \emph{Affine Differential Geometry: Geometry of Affine Immersions}, Cambridge Tracts in Mathematics \textbf{111}, Cambridge University Press, Cambridge, 1994.

\bibitem{ref18} B. Opozda, \emph{Some inequalities and applications of Simons' type formulas in Riemannian, affine, and statistical geometry}, J. Geom. Anal. \textbf{32} (2022), Article 108.

\bibitem{ref19} S. Bácsó, F. Ilosvay, and B. Kis, \emph{Landsberg spaces with common geodesics}, Publ. Math. Debrecen \textbf{42} (1993), 139--144.

\bibitem{ref20} M. Crampin, \emph{A condition for a Landsberg space to be Berwaldian}, Publ. Math. Debrecen \textbf{93} (2018), 143--155.

\bibitem{ref21} M. Crampin, \emph{Invariant volumes, weakly-Berwald Finsler spaces, and the Landsberg--Berwald problem}, Publ. Math. Debrecen \textbf{100} (2022), 101--118.

\bibitem{ref22} Y. Ichijyō, \emph{Finsler manifolds modeled on a Minkowski space}, J. Math. Kyoto Univ. \textbf{16} (1976), 639--652.

\bibitem{ref23} M. Li, \emph{Equivalence theorems of Minkowski spaces and an application in Finsler geometry}, Acta Math. Sinica (Chin. Ser.) \textbf{62} (2019), 177--190; English version: arXiv:1504.04475v2.

\bibitem{ref24} M. Matsumoto and C. Shibata, \emph{On semi-$C$-reducibility, $T$-tensor $=0$, and $S_4$-likeness of Finsler spaces}, J. Math. Kyoto Univ. \textbf{19} (1979), 301--314.

\bibitem{ref25} X. Mo and L. Zhou, \emph{The curvatures of spherically symmetric Finsler metrics in $\mathbb{R}^n$}, arXiv:1202.4543.

\bibitem{ref26} S. Numata, \emph{On Landsberg spaces of scalar curvature}, J. Korean Math. Soc. \textbf{12} (1975), 97--100.

\bibitem{ref27} M. Xu and S. Deng, \emph{The Landsberg equation of a Finsler space}, Ann. Sc. Norm. Super. Pisa Cl. Sci. (5) \textbf{22} (2021), 31--51.

\bibitem{ref28} S. Zhou, J. Wang, and B. Li, \emph{On a class of almost regular Landsberg metrics}, Sci. China Math. \textbf{62} (2019), 935--960.

\bibitem{ref29} H. Feng, Y. Han, and M. Li, \emph{An equivalence theorem of a class of Minkowski norms and its applications}, Sci. China Math. \textbf{64} (2021), 1429--1446.

\bibitem{ref30} C. Shibata, \emph{On the curvature tensor $R^h{}_{ijk}$ of Finsler spaces of scalar curvature}, Tensor (N.S.) \textbf{32} (1978), 311--317.
\end{thebibliography}
\end{document}